\documentclass[11pt]{article}
\usepackage[margin=1.15in]{geometry}
\usepackage[T1]{fontenc}
\usepackage{lmodern}
\usepackage{microtype}
\usepackage{amsmath,amssymb,amsthm,mathtools}
\usepackage{hyperref}
\hypersetup{
  hidelinks,
  pdftitle={The Derived Algebra of Nonlinear Polynomial Divergence-Free Vector Fields},
  pdfauthor={Chao Ma},
  pdfcreator={Chao Ma},
  pdfsubject={The derived algebra of the nonlinear divergence-free polynomial vector fields in every degree and over every field, with the graded abelianization in positive characteristic as a rational GL(V)-module},
  pdfkeywords={divergence-free vector fields, derived algebra, abelianization, Lie algebras of Cartan type, special Lie algebra, Cartier operator, Cartier isomorphism, de Rham cohomology, Koszul complex, positive characteristic, Frobenius twist, rational GL(V)-module, Hilbert series, minimal generators}
}
\usepackage[nameinlink,noabbrev]{cleveref}
\usepackage[numbers,sort&compress]{natbib}
\numberwithin{equation}{section}
\usepackage{aliascnt}
\newtheorem{theorem}{Theorem}[section]
\newaliascnt{proposition}{theorem}
\newtheorem{proposition}[proposition]{Proposition}
\aliascntresetthe{proposition}
\newaliascnt{lemma}{theorem}
\newtheorem{lemma}[lemma]{Lemma}
\aliascntresetthe{lemma}
\newaliascnt{corollary}{theorem}
\newtheorem{corollary}[corollary]{Corollary}
\aliascntresetthe{corollary}
\theoremstyle{definition}
\newaliascnt{definition}{theorem}

\aliascntresetthe{definition}
\theoremstyle{remark}
\newaliascnt{remark}{theorem}
\newtheorem{remark}[remark]{Remark}
\aliascntresetthe{remark}
\theoremstyle{plain}

\newtheorem{mainthm}{Theorem}

\crefname{mainthm}{Theorem}{Theorems}
\crefname{section}{Section}{Sections}
\crefname{subsection}{Section}{Sections}
\Crefname{mainthm}{Theorem}{Theorems}
\crefname{theorem}{Theorem}{Theorems}
\Crefname{theorem}{Theorem}{Theorems}
\crefname{proposition}{Proposition}{Propositions}
\Crefname{proposition}{Proposition}{Propositions}
\crefname{lemma}{Lemma}{Lemmas}
\Crefname{lemma}{Lemma}{Lemmas}
\crefname{corollary}{Corollary}{Corollaries}
\Crefname{corollary}{Corollary}{Corollaries}
\crefname{definition}{Definition}{Definitions}
\Crefname{definition}{Definition}{Definitions}
\crefname{remark}{Remark}{Remarks}
\Crefname{remark}{Remark}{Remarks}

\newcommand{\Ld}{\mathsf L}
\newcommand{\Hex}{\mathsf E}
\DeclareMathOperator{\Sym}{Sym}
\newcommand{\dR}{\mathrm{dR}}
\newcommand{\GL}{\operatorname{GL}}
\newcommand{\SL}{\operatorname{SL}}

\title{The Derived Algebra of Nonlinear Polynomial\\
Divergence-Free Vector Fields}

\author{Chao Ma\thanks{Independent Researcher, London, United Kingdom.\\
ORCID: 0009-0004-2456-9098.  Email: \texttt{raymond.ma@me.com}.}}
\date{}

\begin{document}
\maketitle

\begin{abstract}
We study the Lie algebra $L_{\ge2}$ of divergence-free polynomial vector
fields on $k^n$, $n\ge3$, with coefficients of degree at least two, graded
by coefficient degree.  Over every field its derived algebra in degree
$d\ge3$ is the space of exact fields, those whose contraction with
the volume form is an exact form, and it is already spanned by brackets
with quadratic fields.  In characteristic zero this is the whole degree-$d$
component.  In characteristic $p>0$ the abelianization is nonzero above
degree two exactly in the degrees $d\ge(p-1)(n-1)$ with $d\equiv1-n\pmod p$,
and Cartier descent identifies it with a Frobenius twist of a rational
$\mathrm{GL}_n$-module, tensored with a power of the determinant.  For
$p\ge5$ each exact component is obtained from the previous one by bracketing
with quadratic fields.
\end{abstract}

\section{Introduction}

Let \(k\) be a field, let \(n\ge3\), and let \(S=k[x_1,\ldots,x_n]\).  The
polynomial vector fields \(D=\sum_i f_i\partial_i\) with
\(\operatorname{div}D=\sum_i\partial_if_i=0\) form a Lie algebra, graded by
the degree of the coefficients.  Write \(\Ld_d\) for its component of degree
\(d\) and
\[
\Ld_{\ge2}=\bigoplus_{d\ge2}\Ld_d
\]
for the subalgebra of fields whose coefficients vanish to order two at the
origin.  In characteristic zero the divergence-free fields form one of the
four classical series of Lie algebras of Cartan type; in characteristic
\(p\) their divided-power truncations give the special algebras of the
classification of simple modular Lie algebras \cite{Strade2004}.  The
subalgebra \(\Ld_{\ge2}\) is also the associated graded Lie algebra of the
congruence filtration on the group of volume-preserving formal
automorphisms of \(k[[x_1,\ldots,x_n]]\) that are tangent to the identity.
The results below are the Lie-theoretic input for the study of this group; for the finite-dimensional algebra \(S(n;\mathbf1)^{(1)}\) the
corresponding group is the unipotent radical of its automorphism group
\citep{MaUnipotent}.

Contraction
with the volume form \(\operatorname{vol}=dx_1\wedge\cdots\wedge dx_n\)
identifies \(\Ld_d\) with the closed \((n-1)\)-forms of coefficient degree
\(d\).  Let \(\Hex_d\subseteq\Ld_d\) be the subspace of exact fields, those \(D\)
for which \(\iota_D\operatorname{vol}\) is exact.

Lichnerowicz studied the Lie algebra of infinitesimal automorphisms of a
unimodular structure \citep{Lichnerowicz1974}, and Donzelli extended
Anders\'en--Lempert theory to Lie algebras of zero-divergence vector fields
on complex affine varieties \citep{Donzelli2021}.  In the smooth setting,
Janssens, Ryvkin and Vizman studied the Lie algebra of exact
divergence-free vector fields and its universal central extension
\citep{JanssensRyvkinVizman2025}.  For polynomial fields, Leites and
Poletaeva gave defining relations and minimal generators
\citep{LeitesPoletaeva1997}, and Bavula determined the automorphism group
\citep{Bavula2017}; generators of simple Lie algebras in positive
characteristic were studied by Bois \citep{Bois2009,Bois2010}.  In positive
characteristic Cartier's operator computes the de Rham cohomology of the
polynomial ring \citep{Cartier1957,Illusie1979,DeligneIllusie1987}, and
J\k{e}drzejewicz gave a Frobenius-residue decomposition of divergence-free
polynomial derivations \citep{Jedrzejewicz2019}.  Here we show that the exact fields are
spanned by brackets with quadratic fields, and so determine the derived
algebra of \(\Ld_{\ge2}\).

\begin{mainthm}[\cref{p1:thm:main}]
\label{p1:thm:A}
For every field \(k\) and every \(d\ge3\),
\[
[\Ld_2,\Ld_{d-1}]=\Hex_d .
\]
In particular \([\Ld_{\ge2},\Ld_{\ge2}]=\bigoplus_{d\ge3}\Hex_d\).
\end{mainthm}

In characteristic zero every closed polynomial form of positive degree is
exact, so \(\Hex_d=\Ld_d\) and the derived algebra is
\(\bigoplus_{d\ge3}\Ld_d\).  In characteristic \(p\) the quotients
\(\Ld_d/\Hex_d\) are de Rham cohomology groups.  Put
\[
d_r=(p-1)(n-1)+pr,
\qquad
W_r=\Sym^r(V^*)\otimes V,
\qquad V=k^n,
\]
and write \(W_r^{(1)}\) for the Frobenius pullback of \(W_r\).

\begin{mainthm}[\cref{p1:thm:module,p1:cor:cartier-coefficient-module}]
\label{p1:thm:B}
Let \(\operatorname{char}k=p>0\) and \(d\ge2\).  Then \(\Ld_d/\Hex_d=0\)
unless \(d=d_r\) for some \(r\ge0\), and
\[
\Ld_{d_r}/\Hex_{d_r}\cong(\det V^*)^{\otimes(p-1)}\otimes W_r^{(1)}
\]
as rational \(\GL(V)\)-modules.
\end{mainthm}

Together, the two theorems give the graded abelianization of
\(\Ld_{\ge2}\): it is \(\Ld_2\) in degree two and the Cartier modules of
\cref{p1:thm:B} above it.  Its Hilbert series and the minimal number of
homogeneous generators follow in
\cref{p1:cor:hilbert,p1:cor:sharp}.  For \(p\ge5\) the exact part is
generated one degree at a time from inside the derived algebra.

\begin{mainthm}[\cref{p1:thm:restricted-source-exactness}]
\label{p1:thm:C}
If \(\operatorname{char}k=p\ge5\), then \(\Hex_d=[\Ld_2,\Hex_{d-1}]\) for
every \(d\ge3\).
\end{mainthm}

The proof of \cref{p1:thm:A} is weightwise.  At a fixed multidegree the de
Rham complex of \(S\) is a Koszul complex \((\Lambda^\bullet V^*,
\lambda\wedge-)\) for a covector \(\lambda\) determined by the multidegree.
On the corresponding weight space, the bracket with a suitable quadratic
field acts as the Koszul contraction by the same \(\lambda\).  Generation
then reduces to linear algebra with two hyperplanes.  Separate arguments are
needed only for \(\lambda=-\sigma\), where \(\sigma=x_1+\cdots+x_n\), and,
in characteristic two, for \(\lambda=\sigma\) at boundary weights.  In low degrees, Theorem~2.69 of Benkart,
Gregory and Premet \citep{BenkartGregoryPremet2009} gives a second proof
(\cref{p1:rem:bgp}).

\Cref{foundation-divergence-free-fields} sets up the dictionary with
differential forms and describes the closed and exact loci weight by
weight.  \Cref{foundation-koszul-reduction} proves \cref{p1:thm:A}.
\Cref{foundation-cartier-abelianization} proves \cref{p1:thm:B,p1:thm:C} and
treats the case \(n=3\).

\section*{Conventions}
Throughout, \(k\) is an arbitrary field, \(n\ge3\), \(V=k^n\), and
\[
S=k[x_1,\ldots,x_n]=\operatorname{Sym}(V^*).
\]
\section{Divergence-free fields and exact forms}
\label{foundation-divergence-free-fields}
Fix the coordinate volume form
\[
\operatorname{vol}=dx_1\wedge\cdots\wedge dx_n.
\]
For \(d\ge0\), write
\[
\Ld_d
=
\left\{
D=\sum_{i=1}^n f_i\partial_i\in S_d\otimes V:
\sum_{i=1}^n\partial_i f_i=0
\right\}.
\]
Thus \(\Ld_d\) is the space of homogeneous divergence-free polynomial
vector fields of coefficient degree \(d\).

\subsection{The divergence-free/de Rham dictionary}
\label{p1:sec:derham-dictionary}

Contraction with the fixed volume form gives a linear isomorphism
\[
S_d\otimes V
\xrightarrow{\ \sim\ }
\Omega^{n-1}_d(S),
\qquad
D\longmapsto \iota_D\operatorname{vol},
\]
where the subscript denotes polynomial coefficient degree.

\begin{proposition}[Contraction and divergence]
\label{p1:prop:contraction-divergence}
For every polynomial vector field \(D\),
\[
d\bigl(\iota_D\operatorname{vol}\bigr)
=
(\operatorname{div}D)\operatorname{vol}.
\]
Consequently, contraction identifies \(\Ld_d\) with the closed
\((n-1)\)-forms of coefficient degree \(d\).
\end{proposition}

\begin{proof}
Cartan's formula gives
\[
\mathcal L_D\operatorname{vol}
=d\bigl(\iota_D\operatorname{vol}\bigr)
+\iota_D(d\operatorname{vol})
=d\bigl(\iota_D\operatorname{vol}\bigr),
\]
since \(d\operatorname{vol}=0\).  Since \(\mathcal L_D dx_j=d(Dx_j)=df_j\), the Leibniz rule applied to the volume form gives
\[
\mathcal L_D\operatorname{vol}
=\left(\sum_{j=1}^n\partial_j f_j\right)\operatorname{vol}
=(\operatorname{div}D)\operatorname{vol}.
\]
\end{proof}

Define the exact subspace
\[
\Hex_d
=
\left\{
D\in\Ld_d:
\iota_D\operatorname{vol}\text{ is exact}
\right\}.
\]
Equivalently, \(\Hex_d\) is the inverse image of
\(d\Omega^{n-2}_{d+1}(S)\subseteq\Omega^{n-1}_d(S)\).

\begin{proposition}[Euler homotopy]
\label{p1:prop:euler-homotopy}
Let \(\omega\in\Omega^{n-1}_d(S)\) be closed.  If the scalar
\(d+n-1\in k\) is nonzero, then \(\omega\) is exact.  Hence
\[
\Ld_d=\Hex_d
\qquad
\text{whenever }d+n-1\ne0\text{ in }k.
\]
\end{proposition}

\begin{proof}
Let
\[
\mathcal E=\sum_{i=1}^n x_i\partial_i
\]
be the Euler vector field.  A form of coefficient degree \(d\) and form
degree \(n-1\) has total Euler degree \(d+n-1\), so
\[
\mathcal L_{\mathcal E}\omega=(d+n-1)\omega.
\]
Cartan's homotopy formula and \(d\omega=0\) therefore give
\[
(d+n-1)\omega
=
\bigl(d\iota_{\mathcal E}+\iota_{\mathcal E}d\bigr)\omega
=
d\bigl(\iota_{\mathcal E}\omega\bigr).
\]
Division by \(d+n-1\) proves the claim.
\end{proof}

Let \(G=\GL(V)\)
act on \(S\) by
\[
(g\cdot f)(v)=f(g^{-1}v)
\]
and on polynomial vector fields by push-forward.  Since
\[
g\cdot\operatorname{vol}=\det(g)^{-1}\operatorname{vol},
\]
naturality of contraction gives
\[
g\cdot(\iota_D\operatorname{vol})
=\iota_{g\cdot D}(g\cdot\operatorname{vol})
=\det(g)^{-1}\iota_{g\cdot D}\operatorname{vol}.
\]
Thus contraction is \(G\)-equivariant after tensoring the vector-field
side with \(\det(V^*)\).

\subsection{Homogeneous weight spaces}
\label{p1:sec:setup}

We use the coordinate multigrading.  For
\(w=(w_1,\ldots,w_n)\in\mathbb Z^n\) and \(1\le i\le n\), put
\[
X_i(w)=x^{w+e_i}\partial_i
\]
whenever every exponent in \(w+e_i\) is nonnegative.  Here \(e_i\) denotes
the \(i\)-th coordinate vector of \(\mathbb Z^n\).

Two conventions govern this multigrading.

First, if \(T\subseteq G\) is the diagonal torus and
\(t=\operatorname{diag}(t_1,\ldots,t_n)\), then
\[
t\cdot X_i(w)
=
\left(\prod_{j=1}^n t_j^{-w_j}\right)X_i(w).
\]
Thus the torus character is \(-w\); we call \(w\) the \emph{weight index}, or weight, of \(X_i(w)\).

Second, whenever an integer tuple
\(m=(m_1,\ldots,m_n)\in\mathbb Z^n\) is paired with a vector in \(V\), or
is regarded as a covector, its coordinates are reduced through
\(\mathbb Z\to k\).  In particular,
\[
\lambda=\lambda_w
:=
\sum_{i=1}^n \overline{w_i+1}\,x_i
\in V^*,
\]
where \(\overline{w_i+1}\) is the image of \(w_i+1\) in \(k\).  Thus, in
characteristic \(p>0\), the equation \(\lambda=0\) means
\[
w_i\equiv-1\pmod p
\qquad(1\le i\le n).
\]

A nonzero weight-indexed subspace of \(S_d\otimes V\) has
\[
\sum_{i=1}^n w_i=d-1.
\]
Moreover, either every \(w_i\) is nonnegative, or exactly one coordinate,
say \(w_s\), equals \(-1\).  Indeed, a basis monomial
\(x^\alpha\partial_i\) has index \(w=\alpha-e_i\); hence no coordinate of
\(w\) can be smaller than \(-1\), and two negative coordinates would make
every \(X_i(w)\) undefined.  If \(w_s=-1\), the weight-indexed subspace is
the line \(kX_s(w)\).

Suppose first that \(w\) is nonnegative.  The coefficient-vector map
\[
V\longrightarrow (S_d\otimes V)[w],
\qquad
a=(a_1,\ldots,a_n)
\longmapsto
\sum_{i=1}^n a_iX_i(w),
\]
is an isomorphism.  Direct differentiation gives
\[
\operatorname{div}
\left(\sum_{i=1}^n a_iX_i(w)\right)
=
(\lambda(a))x^w.
\]
Hence, under this coefficient-vector identification, the closed locus is
\begin{equation}
\label{p1:eq:closed-weight}
L(w)=
\begin{cases}
\ker\lambda,&\lambda\ne0,\\
V,&\lambda=0.
\end{cases}
\end{equation}

The exact locus is governed by the same covector.  Fix the total multidegree
\[
m=w+\mathbf1\in\mathbb Z^n.
\]
For \(I\subseteq\{1,\ldots,n\}\), write \(\mathbf1_I\) for its incidence
vector and \(dx_I\) for the corresponding ordered wedge.  Because
\(w_i\ge0\), every basis element
\[
x^{m-\mathbf1_I}dx_I
\]
is defined.  The fixed-multidegree summand of the de Rham complex is
therefore identified with \(\Lambda^\bullet V^*\), and under this
identification the differential is
\[
\lambda\wedge-.
\]
If \(\lambda\ne0\), choose \(q\in V\) with \(\lambda(q)=1\).  Then
\[
\iota_q(\lambda\wedge-)
+\lambda\wedge\iota_q
=\operatorname{id}_{\Lambda^\bullet V^*},
\]
so the weight complex is acyclic.  If \(\lambda=0\), its differential is
zero.  Under contraction with \(\operatorname{vol}\), the exact locus is
therefore
\begin{equation}
\label{p1:eq:exact-weight}
E(w)=
\begin{cases}
\ker\lambda,&\lambda\ne0,\\
0,&\lambda=0.
\end{cases}
\end{equation}

For the boundary case, suppose \(w_s=-1\).  Then the
weight-indexed vector-field space is \(kX_s(w)\), and its basis field is
automatically divergence-free because its coefficient is independent of
\(x_s\).  The total multidegree \(m=w+\mathbf1\) has \(m_s=0\).  Hence the
fixed-multidegree de Rham complex contains only wedges generated by
\[
U_s^*=\operatorname{span}\{dx_i:i\ne s\},
\]
and is the Koszul complex
\[
\bigl(\Lambda^\bullet U_s^*,\lambda\wedge-\bigr).
\]
Since the \(x_s\)-coordinate of \(\lambda\) is zero, this complex is
acyclic exactly when \(\lambda\ne0\).  Its top-degree line, corresponding
under contraction to \(kX_s(w)\), is therefore exact if and only if
\(\lambda\ne0\).

\section{Quadratic brackets and the Koszul reduction}
\label{foundation-koszul-reduction}

\begin{theorem}[Quadratic generation of the exact subspace]
\label{p1:thm:main}
Let \(k\) be any field and let \(n\ge3\).  Then, for every \(d\ge3\),
\[
[\Ld_2,\Ld_{d-1}]=\Hex_d.
\]
\end{theorem}

\subsection{Quadratic anchors and the raw bracket}

Write
\[
\varepsilon_r:=x_r\in V^*,
\qquad
\sigma:=\varepsilon_1+\cdots+\varepsilon_n\in V^*.
\]

For \(1\le r\le n\) and \(a=(a_1,\ldots,a_n)\in V\), define
\[
D_r(a)
:=
\sum_{i=1}^n a_i x_r x_i\partial_i
\in S_2\otimes V.
\]
Its weight index is \(e_r\).  Direct differentiation gives
\[
\operatorname{div}D_r(a)
=
\bigl(\sigma(a)+a_r\bigr)x_r.
\]
Hence, with
\[
\mu_r:=\sigma+\varepsilon_r\in V^*,
\qquad
K_r:=\ker\mu_r,
\]
we have
\[
D_r(a)\in\Ld_2
\quad\Longleftrightarrow\quad
 a\in K_r.
\]
We call the fields \(D_r(a)\), \(a\in K_r\), the \emph{\(r\)-th anchor}.

Fix a target weight index \(w\), and let
\[
\lambda:=\lambda_w\in V^*
\]
be the field-valued covector defined in Section~\ref{foundation-divergence-free-fields}.
Choose an index \(r\) such that
\[
v:=w-e_r
\]
is again a weight index occurring in \(S_{d-1}\otimes V\).  For a coefficient vector \(b\in V\), write
\[
Y_v(b)=\sum_i b_iX_i(v),
\]
where the sum runs over the indices for which \(X_i(v)\) is defined and all
other coordinates of \(b\) are understood to be zero.  Thus all coordinates
are available when \(v\) is nonnegative, whereas if \(v_s=-1\), the only
possible coefficient direction is \(b\in k e_s\).

\begin{proposition}[Raw quadratic bracket]
\label{p1:prop:raw-quadratic-bracket}
With the notation above,
\[
[D_r(a),Y_v(b)]
=
\sum_i c_iX_i(w),
\qquad
c=(a\cdot v)b-b_ra,
\]
where
\[
a\cdot v:=\sum_i \overline{v_i}\,a_i\in k.
\]
If both inputs are divergence-free, then
\begin{equation}
\label{p1:eq:reduced-bracket}
c
=
\lambda(a)b-\lambda(b)a.
\end{equation}
\end{proposition}

\begin{proof}
For the \(i\)-th component of the bracket, differentiation gives
\[
D_r(a)\bigl(b_i x^{v+e_i}\bigr)
=
 b_i(a\cdot v+a_i)x^{w+e_i},
\]
whereas
\[
Y_v(b)\bigl(a_i x_rx_i\bigr)
=
 a_i(b_r+b_i)x^{w+e_i}.
\]
Their difference is therefore
\[
c_i=(a\cdot v)b_i-b_ra_i.
\]

Now suppose that both inputs are divergence-free.  The first condition is
\(\mu_r(a)=0\), so
\[
\sigma(a)+a_r=0.
\]
Since \(v=w-e_r\), reduction of the integer coordinates to \(k\) gives
\[
a\cdot v
=
\lambda(a)-\sigma(a)-a_r
=
\lambda(a).
\]
The second input has covector
\[
\lambda_v=\lambda-\varepsilon_r,
\]
so its divergence-free condition is
\[
(\lambda-\varepsilon_r)(b)=0,
\]
which is equivalent to \(b_r=\lambda(b)\).  Substitution gives
\eqref{p1:eq:reduced-bracket}.
\end{proof}

Equation~\eqref{p1:eq:reduced-bracket} is the Koszul contraction of
\(a\wedge b\) by \(\lambda\).  More precisely, define
\[
\Phi_\lambda:\Lambda^2V\longrightarrow V,
\qquad
\Phi_\lambda(a\wedge b)
=
\lambda(a)b-\lambda(b)a.
\]
Then \(\lambda\Phi_\lambda=0\), so \(\Phi_\lambda\) takes values in
\(\ker\lambda\), and it is the degree-two differential in the contraction
Koszul strand
\[
\Lambda^2V
\xrightarrow{\ \iota_\lambda\ }
V
\xrightarrow{\ \lambda\ }
k.
\]

\subsection{The linear algebra of one anchor}

\begin{lemma}[Koszul surjectivity]
\label{p1:lem:onto}
If \(\lambda\ne0\), then
\[
\Phi_\lambda:\Lambda^2V\twoheadrightarrow\ker\lambda.
\]
\end{lemma}

\begin{proof}
Choose \(q\in V\) with \(\lambda(q)=1\).  For every
\(c\in\ker\lambda\),
\[
\Phi_\lambda(q\wedge c)=c.
\]
\end{proof}

\begin{lemma}[Mixed wedges of distinct hyperplanes]
\label{p1:lem:hyperplanes}
If \(K,N\subset V\) are distinct hyperplanes, then
\[
\operatorname{span}
\{a\wedge b:a\in K,\ b\in N\}
=
\Lambda^2V.
\]
\end{lemma}

\begin{proof}
Let \(Z=K\cap N\).  Since \(K\ne N\), the intersection has codimension two.
Choose
\[
a_0\in K\setminus N,
\qquad
b_0\in N\setminus K.
\]
Then
\[
V=Z\oplus ka_0\oplus kb_0,
\]
and the allowed mixed wedges contain
\[
\Lambda^2Z,
\qquad
a_0\wedge Z,
\qquad
Z\wedge b_0,
\qquad
k(a_0\wedge b_0).
\]
These four summands span \(\Lambda^2V\).
\end{proof}

For a target covector \(\lambda\), the predecessor attached to the
\(r\)-th anchor has divergence-free coefficient space
\[
N_r:=\ker(\lambda-\varepsilon_r).
\]
If \(\lambda-\varepsilon_r=0\), then \(N_r=V\).  Otherwise \(N_r\) is a
hyperplane.

\begin{lemma}[Hyperplane coincidence]
\label{p1:lem:hyperplane-coincidence}
Suppose \(N_r\) is a hyperplane.  Then
\[
N_r=K_r
\]
if and only if there is a scalar \(c_r\in k^\times\) such that
\begin{equation}
\label{p1:eq:hyperplane-coincidence}
\lambda-\varepsilon_r=c_r\mu_r.
\end{equation}
If \eqref{p1:eq:hyperplane-coincidence} holds for two distinct indices
\(r\ne s\), then
\[
c_r=c_s=-1
\]
and
\begin{equation}
\label{p1:eq:simultaneous-exceptional-lambda}
\lambda=-\sigma.
\end{equation}
\end{lemma}

\begin{proof}
Two nonzero covectors have the same kernel exactly when they are nonzero
scalar multiples of one another, which proves the first assertion.

Now suppose coincidence occurs for \(r\ne s\).  Since \(n\ge3\), choose a
coordinate \(t\notin\{r,s\}\).  Evaluating both coincidence identities on
\(e_t\in V\) gives
\[
c_r=c_s=:c.
\]
Evaluating the identity for the \(s\)-anchor on \(e_r\) gives
\(\lambda(e_r)=c\), while the identity for the \(r\)-anchor gives
\[
\lambda(e_r)-1=2c.
\]
Thus \(c-1=2c\), hence \(c=-1\).  Every coordinate of \(\lambda\) is then
\(-1\), that is, \(\lambda=-\sigma\).
\end{proof}

If \(N_r=V\), then \(K_r\wedge V\) spans \(\Lambda^2V\).  If \(N_r\) and
\(K_r\) are distinct hyperplanes, the same conclusion follows from
\cref{p1:lem:hyperplanes}.  In either case, the \(r\)-anchor realizes all
of \(\ker\lambda\) by \cref{p1:lem:onto}.  Only the coincident case
\(N_r=K_r\) needs separate treatment.

\subsection{Proof of quadratic generation}

Fix \(d\ge3\) and a target weight index \(w\) occurring in \(\Ld_d\).
Thus
\[
\sum_{i=1}^n w_i=d-1.
\]
Put \(\lambda=\lambda_w\).  If \(\lambda=0\), Section~\ref{foundation-divergence-free-fields}
shows that the exact subspace at \(w\) is zero.
Hence assume \(\lambda\ne0\).

\subsubsection{Interior weights}

Assume first that
\[
w_i\ge1
\qquad(1\le i\le n).
\]
Then for every anchor \(r\), the predecessor \(v=w-e_r\) is nonnegative,
and its divergence-free coefficient space is
\[
N_r=\ker(\lambda-\varepsilon_r).
\]
If \(N_r=V\) or \(N_r\ne K_r\), the preceding linear algebra shows that
brackets using the \(r\)-anchor span \(\ker\lambda\), which is the
exact subspace at \(w\).

Only coincident anchors remain.  By
\cref{p1:lem:hyperplane-coincidence}, if two distinct anchors are
coincident, then necessarily
\[
\lambda=-\sigma.
\]
Consequently, unless \(\lambda=-\sigma\), at most one anchor is
coincident.  Since \(n\ge3\), another anchor is available and already spans
\(\ker\lambda\).

Suppose now that
\[
\lambda=-\sigma.
\]
Then \(N_r=K_r\) for every \(r\).  The restriction
\(\lambda|_{K_r}\) is nonzero.  Otherwise \(K_r=\ker\mu_r\) and
\(\ker\lambda\) coincide, so \(\mu_r=\sigma+\varepsilon_r\) is proportional
to \(-\sigma\).  A coordinate away from \(r\) forces the scalar to be
\(-1\), and the \(r\)-th coordinate then gives \(2=1\).

Choose \(q_r\in K_r\) with \(\lambda(q_r)=1\).  For
\(c\in K_r\cap\ker\lambda\),
\[
\Phi_\lambda(q_r\wedge c)=c.
\]
Since \(\lambda=-\sigma\), the intersection is
\[
K_r\cap\ker\lambda
=
\left\{
 c\in V:
 \sum_i c_i=0,
 \ c_r=0
\right\}.
\]
These subspaces span \(\ker\lambda\).  Indeed, for \(i\ne j\), choose
\(r\notin\{i,j\}\); then
\[
e_i-e_j\in K_r\cap\ker\lambda,
\]
and the vectors \(e_i-e_j\) span the hyperplane
\(\sum_i c_i=0\).  Thus the
simultaneous exceptional configuration is generated as well.

\subsubsection{Nonnegative boundary weights}

Assume next that \(w\) is nonnegative and that
\[
w_r=0
\]
for some \(r\).  Then \(\lambda(e_r)=1\), while the predecessor
\(v=w-e_r\) has its unique negative coordinate at \(r\).  Its
coefficient space is therefore the line \(k\,e_r\).  Taking
\(b=e_r\) in \eqref{p1:eq:reduced-bracket} gives the linear map
\begin{equation}
\label{p1:eq:boundary-anchor-map}
T_r:K_r\longrightarrow\ker\lambda,
\qquad
T_r(a)=\lambda(a)e_r-a.
\end{equation}

Suppose first that \(\operatorname{char}k\ne2\).  If \(T_r(a)=0\), then
\(a\in k\,e_r\).  But
\[
\mu_r(e_r)=2\ne0,
\]
so \(e_r\notin K_r\) and hence \(a=0\).  Thus \(T_r\) is injective.
Both \(K_r\) and \(\ker\lambda\) have dimension \(n-1\), so \(T_r\) is
an isomorphism.

It remains to treat \(\operatorname{char}k=2\).  Since
\[
\sum_i w_i=d-1\ge2,
\]
there is an index \(s\ne r\) with \(w_s>0\).  The \(s\)-predecessor is
nonnegative and has coefficient space
\[
N_s=\ker(\lambda-\varepsilon_s).
\]
If \(N_s=V\) or \(N_s\ne K_s\), the generic anchor argument already spans
\(\ker\lambda\).  We may therefore assume \(N_s=K_s\), so
\[
\lambda-\varepsilon_s=c\mu_s
\]
for some \(c\ne0\).  Evaluating on \(e_r\) gives \(c=1\), because
\(\lambda(e_r)=1\).  Evaluating on \(e_s\) and using \(2=0\) then gives
\(\lambda(e_s)=1\); every remaining coordinate is also equal to \(1\).
Hence
\begin{equation}
\label{p1:eq:char2-boundary-lambda}
\lambda=\sigma
\qquad
(\operatorname{char}k=2).
\end{equation}

For \(u\in\{r,s\}\), put
\[
H_u
:=
\left\{
 c\in V:
 \sigma(c)=0,
 \ c_u=0
\right\}.
\]
Because \(\lambda=\sigma\), this is
\(H_u=\ker\lambda\cap\ker\varepsilon_u\).
For \(a\in K_r\), the relation \(\mu_r(a)=0\) reads
\[
\sigma(a)+a_r=0.
\]
In characteristic two this gives \(\lambda(a)=a_r\), so
\[
T_r(a)=a_re_r-a\in H_r.
\]
Moreover
\[
\ker T_r=k\,e_r
\]
because \(e_r\in K_r\) when \(2=0\).  Therefore
\[
\operatorname{im}T_r=H_r.
\]

For the \(s\)-anchor, the assumption \(N_s=K_s\) means that both
coefficient vectors range over \(K_s\).  Since \(e_s\in K_s\) and
\(\lambda(e_s)=1\),
\[
\Phi_\lambda(e_s\wedge c)=c
\qquad(c\in H_s),
\]
so \(H_s\) lies in the \(s\)-anchor image.  Conversely, for
\(a,b\in K_s\), the equality
\(\lambda(a)=a_s\) and its analogue for \(b\) show that the \(s\)-th
coordinate of \(\Phi_\lambda(a\wedge b)\) is zero.  Hence the
\(s\)-anchor image is exactly \(H_s\).

Finally,
\[
\dim H_r=\dim H_s=n-2,
\qquad
\dim(H_r\cap H_s)=n-3.
\]
Therefore
\[
H_r+H_s=\ker\lambda.
\]
This completes the characteristic-two boundary case.

\subsubsection{Weights with one negative coordinate}

Finally suppose that
\[
w_s=-1
\]
for a unique index \(s\).  The target weight space is the line
\[
kX_s(w).
\]
Since
\[
\sum_{j\ne s}w_j=d\ge3,
\]
there is an index \(r\ne s\) with \(w_r>0\).  The predecessor
\(v=w-e_r\) has the same unique negative coordinate \(s\), so its
coefficient space is the line \(k\,e_s\).  Equation
\eqref{p1:eq:reduced-bracket} gives
\[
\Phi_\lambda(a\wedge e_s)=\lambda(a)e_s
\qquad(a\in K_r),
\]
because \(\lambda(e_s)=0\).

The covectors \(\lambda\) and \(\mu_r\) are not proportional: their
\(s\)-coordinates are respectively \(0\) and \(1\), since \(r\ne s\).
Consequently \(\lambda|_{K_r}\ne0\).  Choosing
\(a\in K_r\) with \(\lambda(a)\ne0\) therefore produces a nonzero multiple
of \(e_s\), and the brackets span the target line.  By the boundary
calculation of Section~\ref{foundation-divergence-free-fields}, this line is
the exact subspace for \(\lambda\ne0\).

\subsubsection{Brackets are exact}

It remains to show that every bracket of divergence-free fields is exact.
Let \(D\) and \(E\) be divergence-free polynomial vector fields.  Since
\[
\mathcal L_D\operatorname{vol}=0,
\qquad
\mathcal L_E\operatorname{vol}=0,
\qquad
d\operatorname{vol}=0,
\]
the Cartan commutator identity gives
\[
\begin{aligned}
\iota_{[D,E]}\operatorname{vol}
&=
[\mathcal L_D,\iota_E]\operatorname{vol}\\
&=
\mathcal L_D(\iota_E\operatorname{vol})\\
&=
(d\iota_D+\iota_Dd)(\iota_E\operatorname{vol})\\
&=
d(\iota_D\iota_E\operatorname{vol}),
\end{aligned}
\]
because
\[
d(\iota_E\operatorname{vol})
=\mathcal L_E\operatorname{vol}=0.
\]
Thus
\[
[\Ld_2,\Ld_{d-1}]\subseteq\Hex_d.
\]
The three weight cases above give the opposite inclusion in every weight
index, and hence
\[
[\Ld_2,\Ld_{d-1}]=\Hex_d.
\]
This proves \cref{p1:thm:main}.

\begin{remark}[Low degrees]
\label{p1:rem:bgp}
Benkart, Gregory and Premet determined, for \(\ell\le p-2\), the
degree-\(\ell\) component of the divergence-free subalgebra of the
divided-power Witt--Jacobson algebra as a module over its degree-zero
component, which corresponds to \(\Ld_{\ell+1}\) here
\citep[Theorem~2.69]{BenkartGregoryPremet2009}: it is irreducible
off \(n+\ell\equiv0\pmod p\) and indecomposable at resonance.  In the range
\(2\le\ell\le p-2\) it is therefore generated by a highest-weight vector.  The passage from divided powers to polynomials is possible
because \(p\)-th powers are killed by all partial derivatives, and the
generator then becomes a commutator: for \(j\notin\{1,n\}\),
\[
 [x_1^2\partial_j,\,
  x_1^{\ell-1}x_j\partial_n]
 =x_1^{\ell+1}\partial_n.
\]
By the Jacobi identity \([\Ld_2,\Ld_{d-1}]\) is stable under \(\Ld_1\), and
\(\Hex_d=\Ld_d\) for \(d\le p-1\) by \cref{p1:cor:cartier-coefficient-module}.
Hence Theorem~2.69 gives \cref{p1:thm:main} in coefficient degrees
\(3\le d\le p-1\), including resonance.
\end{remark}

\section{Cartier abelianization and homogeneous generators}
\label{foundation-cartier-abelianization}

By the Cartan
calculation at the end of Section~\ref{foundation-koszul-reduction}, every
bracket of divergence-free fields is exact.  Since the least bracket degree
inside \(\Ld_{\ge2}\) is three, Theorem~\ref{p1:thm:main} gives
\[
[\Ld_{\ge2},\Ld_{\ge2}]_2=0,
\qquad
[\Ld_{\ge2},\Ld_{\ge2}]_d=\Hex_d
\quad(d\ge3).
\]
Consequently
\[
(\Ld_{\ge2}^{\mathrm{ab}})_2=\Ld_2,
\qquad
(\Ld_{\ge2}^{\mathrm{ab}})_d=\Ld_d/\Hex_d
\quad(d\ge3).
\]
The remaining quotient \(\Ld_d/\Hex_d\) is determined by de Rham
cohomology.  

Throughout this section, when \(\operatorname{char}k=p>0\), write
\[
G=\GL(V),
\qquad
V^{(1)}:=k\otimes_{F_k,k}V,
\qquad
G^{(1)}=\GL(V^{(1)}),
\]
and let
\[
F_{G/k}:G\longrightarrow G^{(1)}
\]
be the relative Frobenius morphism.  If \(M\) is a rational
\(G^{(1)}\)-module, then \(F_{G/k}^*M\) denotes restriction of its action
along \(F_{G/k}\).  This is the scheme-theoretic Frobenius pullback; over a
perfect field it is the usual first Frobenius twist.

\subsection{The Cartier module}
\label{p1:sec:module}

\begin{theorem}[Graded abelianization and Cartier module]
\label{p1:thm:module}
Assume \(\operatorname{char}k=p>0\).  For every \(d\ge2\), contraction with
\(\operatorname{vol}\) induces a natural isomorphism of rational
\(G=\GL(V)\)-modules
\[
(\Ld_d/\Hex_d)\otimes\det(V^*)
\;\cong\;
H^{n-1}_{\dR}(S)_{d+n-1},
\]
where the de Rham complex is graded by
\(\deg x_i=\deg dx_i=1\).

The right-hand side vanishes unless
\[
d=pq-(n-1)
\qquad\text{for some }q\ge n-1.
\]
Writing \(r=q-n+1\ge0\), one has
\[
H^{n-1}_{\dR}(S)_{pq}
\cong
F_{G/k}^*\!\left(
\Sym^r((V^{(1)})^*)
\otimes
\Lambda^{n-1}((V^{(1)})^*)
\right),
\]
and therefore
\[
\dim_k(\Ld_d/\Hex_d)
=n\binom{q}{n-1}.
\]
\end{theorem}

\begin{proof}
The first isomorphism follows from the divergence-free/de Rham dictionary of
Section~\ref{foundation-divergence-free-fields}, including its
determinant twist.

Let \(X=\operatorname{Spec}S\), let
\(X^{(1)}=X\times_{\operatorname{Spec}k,F_k}\operatorname{Spec}k\), and
let
\[
F_{X/k}:X\longrightarrow X^{(1)}
\]
be relative Frobenius.  On coordinate rings its comorphism sends
\(x_i^{(1)}\) to \(x_i^p\).  The relative inverse Cartier isomorphism is
\[
C^{-1}:\Omega^{n-1}_{X^{(1)}/k}
\xrightarrow{\ \sim\ }
\mathcal H^{n-1}(F_{X/k*}\Omega^\bullet_{X/k}).
\]
It is natural for linear changes of coordinates.  For a multi-index
\(\alpha\in\mathbf Z_{\ge0}^n\) and a subset
\(I\subset\{1,\ldots,n\}\) of cardinality \(n-1\), its coordinate form is
\[
C^{-1}\!\left((x^{(1)})^\alpha dx_I^{(1)}\right)
=
\left[
 x^{p\alpha+(p-1)\mathbf1_I}dx_I
\right],
\]
where the brackets denote the de Rham cohomology class and
\(\mathbf1_I\) is the indicator vector of \(I\).

The source form has total degree
\[
q=|\alpha|+(n-1),
\]
whereas its Cartier representative has total degree
\[
p|\alpha|+(p-1)(n-1)+(n-1)=pq.
\]
Hence
\[
H^{n-1}_{\dR}(S)_m=0
\quad\text{unless}\quad
m=pq\text{ with }q\ge n-1.
\]
For \(m=pq\), the degree-\(q\) component of
\(\Omega^{n-1}_{X^{(1)}/k}\) is
\[
\Sym^{q-n+1}((V^{(1)})^*)
\otimes
\Lambda^{n-1}((V^{(1)})^*).
\]
Cartier naturality identifies its image with the pullback of this
\(G^{(1)}\)-module along \(F_{G/k}\).  Since the form degree corresponding
to a coefficient-degree-\(d\) vector field is \(d+n-1\), the grading
support condition is
\[
d+n-1=pq,
\qquad\text{equivalently}\qquad
d=pq-(n-1).
\]
Finally, with \(r=q-n+1\),
\[
\dim\Sym^r((V^{(1)})^*)=\binom{q}{n-1},
\qquad
\dim\Lambda^{n-1}((V^{(1)})^*)=n,
\]
which gives the stated dimension.
\end{proof}

Removing the determinant line introduced by contraction gives the following
equivalent formulation.

\begin{corollary}[Cartier coefficient modules]
\label{p1:cor:cartier-coefficient-module}
Put
\[
d_r=(p-1)(n-1)+pr,
\qquad
W_r=\Sym^r(V^*)\otimes V,
\qquad
\delta_1=(\det V^*)^{\otimes(p-1)}.
\]
Define the first Frobenius pullback of \(W_r\) by
\[
W_r^{(1)}
:=
F_{G/k}^*\!\left(
\Sym^r((V^{(1)})^*)\otimes V^{(1)}
\right).
\]
Then
\[
{
\Ld_{d_r}/\Hex_{d_r}
\cong
\delta_1\otimes W_r^{(1)}
}
\qquad(r\ge0)
\]
as rational \(G\)-modules, and \(\Ld_d/\Hex_d=0\) in every other degree
\(d\ge2\).
\end{corollary}

\begin{proof}
For every \(n\)-dimensional vector space \(U\), contraction with a top
exterior form gives the canonical isomorphism
\[
\Lambda^{n-1}(U^*)\cong U\otimes\det(U^*).
\]
Apply this with \(U=V^{(1)}\).  Pullback along relative Frobenius gives
\[
F_{G/k}^*\det((V^{(1)})^*)
\cong
(\det V^*)^{\otimes p},
\]
because the determinant character is sent to its \(p\)-th power.  Hence
Theorem~\ref{p1:thm:module} gives
\[
(\Ld_{d_r}/\Hex_{d_r})\otimes\det(V^*)
\cong
W_r^{(1)}\otimes(\det V^*)^{\otimes p}.
\]
Cancelling the determinant line yields the displayed formula.
\end{proof}

\begin{remark}[Specialization to the finite linear group]
When \(k=\mathbf F_p\) and the rational representation is restricted to
\(\GL_n(\mathbf F_p)\), the character \(\delta_1\) is trivial because
\(a^{p-1}=1\) for \(a\in\mathbf F_p^\times\).  The Frobenius-twisted and
untwisted matrices of \(W_r\) also agree on \(\GL_n(\mathbf F_p)\).
\end{remark}

\subsection{Hilbert series and homogeneous generators}

\begin{corollary}[Hilbert series]
\label{p1:cor:hilbert}
Except when \((n,p)=(3,2)\), the Hilbert series of
\(\Ld_{\ge2}^{\mathrm{ab}}\), graded by coefficient degree, is
\[
h_{n,p}(t)
=
\frac{n(n-1)(n+2)}{2}\,t^2
+n\,\frac{t^{(p-1)(n-1)}}{(1-t^p)^n}.
\]
For \((n,p)=(3,2)\),
\[
h_{3,2}(t)
=12t^2+\frac{3t^2}{(1-t^2)^3}.
\]
\end{corollary}

\begin{proof}
The divergence map \(S_2\otimes V\to S_1\) is surjective: for every
\(j\), choose \(i\ne j\) and observe that
\[
\operatorname{div}(x_ix_j\partial_i)=x_j.
\]
Therefore
\[
\dim\Ld_2
=n\binom{n+1}{2}-n
=\frac{n(n-1)(n+2)}2.
\]

Write
\[
d_r=(p-1)(n-1)+pr.
\]
By the preceding corollary, the Cartier quotient in degree \(d_r\) has
dimension
\[
n\binom{r+n-1}{n-1}.
\]
Thus its generating series is
\[
\sum_{r\ge0}n\binom{r+n-1}{n-1}t^{d_r}
=
n\,\frac{t^{(p-1)(n-1)}}{(1-t^p)^n}.
\]
For \(n\ge3\), the initial Cartier degree
\[
d_0=(p-1)(n-1)
\]
is at least two, and \(d_0=2\) occurs exactly for \((n,p)=(3,2)\).  If
\(d_0\ge3\), the Cartier series is disjoint from the degree-two
abelianization \(\Ld_2\), giving the first formula.

For \((n,p)=(3,2)\), the Cartier series begins with a three-dimensional
term in degree two, namely \(\Ld_2/\Hex_2\), whereas the degree-two
abelianization is the whole \(\Ld_2\), of dimension fifteen, because no
bracket in \(\Ld_{\ge2}\) has degree two.  Replacing that
three-dimensional term by the fifteen-dimensional degree-two piece adds
\(12t^2\) and gives the second formula.
\end{proof}

\begin{corollary}[Minimal homogeneous generators]
\label{p1:cor:sharp}
Let \(M\subset\Ld_{\ge2}\) be a homogeneous subspace whose Lie closure is
\(\Ld_{\ge2}\).  Then, for every \(d\), the composite
\[
M_d\longrightarrow\Ld_d
\longrightarrow(\Ld_{\ge2}^{\mathrm{ab}})_d
\]
is surjective; in particular
\(\dim M_d\ge\dim(\Ld_{\ge2}^{\mathrm{ab}})_d\).

These lower bounds are sharp.  One may take
\(M_2=\Ld_2\) and, for each Cartier degree
\(d=pq-(n-1)\ge3\), choose any vector-space complement
\[
\Ld_d=\Hex_d\oplus C_d
\]
and include \(C_d\) in \(M\); no additional homogeneous generators are
needed.
\end{corollary}

\begin{proof}
Every generating subspace maps onto the abelianization, which proves the
lower bound degree by degree.

For sharpness, proceed by induction on coefficient degree.  Degree two is
present by construction.  Suppose all layers below degree \(d\) have
already been generated.  Then Theorem~\ref{p1:thm:main} gives
\[
\Hex_d=[\Ld_2,\Ld_{d-1}],
\]
so the exact subspace is generated from earlier layers.  If \(d\) is
not a Cartier degree, Theorem~\ref{p1:thm:module} gives
\(\Ld_d=\Hex_d\).  At a Cartier degree, the chosen complement \(C_d\)
supplies the remaining quotient \(\Ld_d/\Hex_d\).  The induction therefore
generates every homogeneous layer.
\end{proof}

\subsection{Restricted-source recursion}

In \cref{p1:thm:main} the second argument of the bracket ranges over all
of \(\Ld_{d-1}\).  For \(p\ge5\) it can be taken in the exact subspace
\(\Hex_{d-1}\).

\begin{theorem}[Restricted-source exactness]
\label{p1:thm:restricted-source-exactness}
Assume \(\operatorname{char}k=p\ge5\) and \(n\ge3\).  Then
\[
{
\Hex_d=[\Ld_2,\Hex_{d-1}]
}
\qquad(d\ge3).
\]
\end{theorem}

\begin{proof}
For distinct coordinates \(i,j\), put
\[
D_{ij}(f)=f_{x_j}\partial_i-f_{x_i}\partial_j.
\]
Under contraction with \(\operatorname{vol}\), the fields \(D_{ij}(f)\)
are the elementary exact fields coming from monomial
\((n-2)\)-form potentials; consequently they span every \(\Hex_d\).
It is therefore enough to treat a monomial potential for one fixed pair of
coordinates.  Rename that pair \(x,y\), choose a third coordinate \(z\),
and write
\[
f=x^Ay^Bz^CM,
\]
where \(M\) is independent of \(x,y,z\).  Scalar expressions involving \(A,B,C\) below are interpreted through \(\mathbf Z\to k\).

For integers \(a,b\ge0\) with \(a+b=3\), whenever all displayed exponents
are nonnegative, a direct calculation gives
\begin{equation}
\label{p1:eq:exact-plane-source}
\begin{aligned}
&[D_{xy}(x^ay^b),
  D_{xy}(x^{A-a+1}y^{B-b+1}z^CM)]\\
&\hspace{35mm}
=\bigl(b(A+1)-a(B+1)\bigr)D_{xy}(f).
\end{aligned}
\end{equation}
The first entry has coefficient degree two, and the second is an exact
field of coefficient degree one less than the target.

\smallskip
\noindent\emph{Cases outside the residue class \(A\equiv B\equiv-1\pmod p\).}
Suppose first that \(A,B\ge1\).  The choices
\((a,b)=(2,1)\) and \((1,2)\) are both admissible.  If both corresponding
scalars vanish, then
\[
(A+1)-2(B+1)=0,
\qquad
2(A+1)-(B+1)=0,
\]
so
\[
3(A+1)=3(B+1)=0.
\]
Because \(p\ge5\), this is equivalent to
\[
A\equiv B\equiv-1\pmod p.
\]
Thus outside that residue class, one of the two brackets in
\eqref{p1:eq:exact-plane-source} produces \(D_{xy}(f)\).

Now suppose \(A=0\).  If \(B=0\), then \(D_{xy}(f)=0\).  If \(B\ge2\),
the admissible choice \((a,b)=(0,3)\) has scalar \(3\ne0\), so again
\eqref{p1:eq:exact-plane-source} applies.  The only remaining case is
\(A=0,B=1\).  The case \(B=0,A=1\) is symmetric.

For \(A=0,B=1\), write
\[
f=yN,
\qquad
D_{xy}(f)=N\partial_x.
\]
The target has coefficient degree \(d\ge3\), hence the monomial \(N\) has
positive degree.  Choose the third coordinate \(z\) from its support and
write
\[
N=z^CM_0,
\qquad
C>0,
\]
with \(M_0\) independent of \(x,y,z\).  Set
\[
U=z^2\partial_x\in\Ld_2,
\qquad
V=D_{yz}\!\left(\frac12 yz^{C-1}M_0\right)\in\Hex_{d-1}.
\]
Explicitly,
\[
V=
-\frac12z^{C-1}M_0\partial_z
+\frac{C-1}{2}yz^{C-2}M_0\partial_y,
\]
where the second summand is absent when \(C=1\).  A direct bracket
calculation gives
\[
[U,V]=N\partial_x.
\]
This proves the remaining case.

\smallskip
\noindent\emph{The residue \(A\equiv B\equiv-1\pmod p\).}
Here \(A,B\ge p-1\ge4\), so all powers of \(x\) and \(y\) appearing below
are nonnegative.  Put
\[
\begin{aligned}
U_1&=z^2\partial_x,
&
V_1&=D_{yz}(x^Ay^Bz^{C-1}M),\\
U_2&=yz\partial_x,
&
V_2&=D_{yz}(x^Ay^{B-1}z^CM).
\end{aligned}
\]
For \(C>0\), these are polynomial fields; when \(C=1\), the derivative
term carrying the factor \(C-1\) is absent.  Using
\(B=-1\) in \(k\), their expanded forms are
\[
\begin{aligned}
V_1&=(C-1)x^Ay^Bz^{C-2}M\partial_y
     +x^Ay^{B-1}z^{C-1}M\partial_z,\\
V_2&=Cx^Ay^{B-1}z^{C-1}M\partial_y
     +2x^Ay^{B-2}z^CM\partial_z.
\end{aligned}
\]
Thus \(V_1,V_2\in\Hex_{d-1}\), while \(U_1,U_2\in\Ld_2\).
If \(C>0\) and \(C-2\ne0\) in \(k\), then
\[
[U_2,V_2]-2[U_1,V_1]
=(C-2)D_{xy}(f),
\]
and division by \(C-2\) gives the target.

If \(C>0\) and \(C\equiv2\pmod p\), set
\[
U_3=z^2\partial_y,
\qquad
V_3=D_{xz}(x^Ay^Bz^{C-1}M).
\]
Since \(A=-1\) in \(k\),
\[
V_3=(C-1)x^Ay^Bz^{C-2}M\partial_x
+x^{A-1}y^Bz^{C-1}M\partial_z,
\]
so again \(V_3\in\Hex_{d-1}\), and direct calculation gives
\[
D_{xy}(f)=[U_1,V_1]-[U_3,V_3].
\]

Finally suppose \(C=0\).  Put
\[
R=x^Ay^{B-2}M\partial_z
=\frac12D_{yz}(x^Ay^{B-1}M)\in\Hex_{d-1}
\]
and
\[
T=D_{yz}(x^Ay^{B-2}zM)
=x^Ay^{B-2}M\partial_y
+3x^Ay^{B-3}zM\partial_z
\in\Hex_{d-1}.
\]
Then
\[
D_{xy}(f)
=
3[zy\partial_x,R]-[y^2\partial_x,T].
\]
All scalar denominators used above are invertible because \(p\ge5\).
This exhausts the monomial exact generators and proves
\[
\Hex_d\subseteq[\Ld_2,\Hex_{d-1}].
\]
The reverse inclusion follows from the Cartan identity of
Section~\ref{foundation-koszul-reduction}, since every bracket of two
divergence-free fields is exact.  Hence equality holds.
\end{proof}

\subsection{The case \texorpdfstring{$n=3$}{n=3}}
\label{p1:sec:n3}

When \(n=3\), the canonical identity
\[
\Lambda^2(V^*)\cong V\otimes\det(V^*)
\]
simplifies the coefficient module.  For
\(d=pq-2\ge3\), put \(r=q-2\ge0\).  Theorem~\ref{p1:thm:module} gives
\[
(\Ld_d/\Hex_d)\otimes\det(V^*)
\cong
F_{G/k}^*\!\left(
\Sym^r((V^{(1)})^*)
\otimes
\Lambda^2((V^{(1)})^*)
\right).
\]
Equivalently, Corollary~\ref{p1:cor:cartier-coefficient-module} gives
\[
\Ld_d/\Hex_d
\cong
(\det V^*)^{\otimes(p-1)}\otimes W_r^{(1)}.
\]
After restriction to \(\SL_3\), all determinant characters disappear, so
\[
(\Ld_d/\Hex_d)|_{\SL_3}
\cong
F_{\SL_3/k}^*\!\left(
\Sym^r((V^{(1)})^*)\otimes V^{(1)}
\right),
\]
and
\[
\dim_k(\Ld_d/\Hex_d)=3\binom{q}{2}.
\]

\section*{Acknowledgments}

The author thanks Helmut Strade and Alexander Premet for helpful
comments, and in particular Alexander Premet for pointing out the
connection with Theorem~2.69 of Benkart, Gregory, and Premet
\citep{BenkartGregoryPremet2009}.

\bibliographystyle{plainnat}

\end{document}